\documentclass[12pt]{article}
\pdfoutput=1
\usepackage{Max}
\usepackage{tikz-cd}

\usepackage[shortlabels]{enumitem}

\usepackage{mathrsfs}
\usepackage{tcolorbox}
\definecolor{lightgrey}{rgb}{0.9,0.9,0.9}
\definecolor{lightblue}{rgb}{.8,.9,.95}
\definecolor{lightgreen}{rgb}{.7,1,.8}
\definecolor{lightorange}{rgb}{1,.9,.7}
\usepackage{empheq}

\title{A Concrete Variant of the Twistor Theorem}
\author[1]{Laura Fredrickson\thanks{lfredric@uoregon.edu}}
\author[2]{Max Zimet\thanks{maxzimet@gmail.com}}
\affil[1]{Department of Mathematics,

University of Oregon, Eugene, OR 97403 USA}

\affil[2]{The Voleon Group,

Berkeley, CA 94704 USA}

\date{}

\definecolor{coolcolor}{rgb}{0.6,0,1}

\definecolor{greencolor}{rgb}{0,0.6,0.5}

  \theoremstyle{plain}
 
\newinttheorem{cor}{Corollary}
\newinttheorem{lem}{Lemma}
\newinttheorem{prop}{Proposition}
\newinttheorem{conj}{Conjecture}
 \newinttheorem{axiom}{Axiom}
 \newinttheorem{ass}{Assumption}
 \theoremstyle{definition}
 \newinttheorem{defn}{Definition}
 \newinttheorem{notation}{Notation}
 \newinttheorem{ex}{Example}
 \newinttheorem{constr}{Construction} 
 \newtheorem*{thm*}{Theorem}
 \theoremstyle{remark}
 \newinttheorem{rmk}{Remark}
 \newinttheorem{rem}{Remark}
 
\def\beq{\begin{eqnarray}}
\def\eeq{\end{eqnarray}}
 \newcommand{\bp}{\begin{proof}[Proof]}
 \newcommand{\ep}{\end{proof}}

\newcommand\nc{\newcommand}
\nc\mc\mathcal
\nc\mb\mathbb
\nc\Up\Upsilon

 \DeclareFontFamily{U}{wncy}{}
    \DeclareFontShape{U}{wncy}{m}{n}{<->wncyr10}{}
    \DeclareSymbolFont{mcy}{U}{wncy}{m}{n}
    \DeclareMathSymbol{\Sh}{\mathord}{mcy}{"58} 

\usepackage{tocloft}

\begin{document}

\maketitle

\begin{abstract}
In this note, we prove a concrete variant of the twistor theorem of Hitchin--Karlhede--Lindstr\"om--Ro\v{c}ek which applies when one already has the real manifold on which one wishes to construct a hyper-K\"ahler structure, and so one does not need to construct it as a parameter space of twistor lines. 
\end{abstract}

\hypersetup{linkcolor=blue}

\section{Introduction}\label{sec:introduction}

Given a hyper-K\"ahler manifold $\M$ of real dimension $4n$ one can associate a twistor manifold $\mathcal{Z}$ of real dimension $4n+2$ that as a complex manifold encodes the $\mathbb{P}^1$-family of complex structures on $\M$.  The hyper-K\"ahler structure gives a $\mathbb{P}^1$-family of holomorphic symplectic structures $(\M, J^{(\zeta)}, \varpi^{(\zeta)})$ on $\M$, and in turn $\mathcal{Z}$ carries a twisted fiberwise holomorphic symplectic form $\varpi(\zeta)$.
Conversely, 
given a twistor manifold $\mathcal{Z}$---including the data of $\varpi(\zeta)$---, the celebrated twistor theorem of Hitchin--Karlhede--Lindstr\"om--Ro\v{c}ek \cite{hitchin:hkSUSY} produces a pseudo-hyper-K\"ahler metric on the space of all real holomorphic sections of $\mathcal{Z}$ having normal bundle isomorphic to $\mathcal{O}(1)^{\oplus 2r}$. This is a powerful way of producing pseudo-hyper-K\"ahler metrics.

\bigskip

We prove a concrete variant of the twistor theorem that applies when one already has the real manifold $\M$ that one wishes to construct a pseudo-hyper-K\"ahler structure on.  We will start with the data of a $\CC^\times$-family of holomorphic symplectic forms on the real manifold $\M$ (see Definition \ref{def:holoSympDef}) of the shape
\be \varpi(\zeta) = - \frac{i}{2\zeta} \omega_+ + \omega_3 - \frac{i}{2} \zeta \omega_- \ee 
where $\omega_-=\bar \omega_+$ and $\bar \omega_3=\omega_3$ and $\omega_+$ is also a holomorphic symplectic form on the real manifold $\M$. 
Given this data, we prove there is a pseudo-hyper-K\"ahler structure on $\M$ (Theorem \ref{thm:twist}b). Moreover, we prove that if $\varpi(\zeta)$ are extracted from a pseudo-hyper-K\"ahler structure on $\M$, the construction recovers the original pseudo-hyper-K\"ahler structure (Theorem \ref{thm:twist}c). 

Giving slightly more detail, in the statement of Hitchin--Karlhede--Lindstr\"om--Ro\v{c}ek's twistor theorem, one begins with a twistor space $p: \mathcal{Z} \to \mathbb{P}^1$ with a real structure carrying a compatible twisted fiberwise holomorphic symplectic form $\varpi(\zeta)$, which is a holomorphic section of  
 $\wedge^2 T^*_{(1,0)} \M \otimes p^*\mathcal{O}(2) \to \Z$;  their theorem produces a pseudo-hyper-K\"ahler metric on the space of all real holomorphic sections of $\mathcal{Z}$ having normal bundle isomorphic to $\mathcal{O}(1)^{\oplus 2r}$.  If one instead begins with the above data on $\M$ and makes the twistor manifold, it was not clear to us that the sections corresponding to points of $\M$ automatically have normal bundle isomorphic to $\mathcal{O}(1)^{\oplus 2r}$ or whether an additional assumption was necessary.  
 Our proof shows that the normal bundle condition is satisfied automatically!  The concrete assumptions on $\varpi(\zeta)$ above imply the  induced complex structures $J^{(\zeta)}$ on $\M$ fit together nicely in the sense that there is an isomorphism $\kappa: T^{(0,1)}\M^{(0)}\to T^{(1,0)}\M^{(0)}$ such that $T^{(0,1)} \M^{(\zeta)} = (1+\zeta \kappa) T^{(0,1)} \M^{(0)}$ for all $\zeta\in\CC$, where $\M^{(\zeta)}$ is the complex manifold $(\M, J^{(\zeta)})$.  This isomorphism implies the normal bundle condition.
 
 Rather than appeal to \cite[Theorem 3.3]{hitchin:hkSUSY} and prove that the condition on the normal bundle is satisfied, we take the opportunity to translate the proof in \cite{hitchin:hkSUSY} to the real manifold $\M$ itself. It is delightfully straightforward!  After all, a hard part of the twistor theorem is the construction of the real manifold. Since we begin with $\M$, there is no sense in constructing it again. Moreover, it is now very clear that if one applies Hitchin--Karlhede--Lindstr\"om--Ro\v{c}ek's twistor theorem or our concrete variant of it to a twistor space $\Z$ obtained from a pseudo-hyper-K\"ahler manifold that one recovers the original hyper-K\"ahler structure (Theorem \ref{thm:twist}c, Corollary \ref{cor:match}).  

\bigskip

Finally, we explain our own motivation.  For us, this is a preliminary note in a series of papers making rigorous Gaiotto--Moore--Neitzke's conjectural machine for translating enumerative data (as well as data describing the semi-flat limit) into explicit smooth hyper-K\"ahler manifolds near semi-flat limits. The hyper-K\"ahler metrics that arise in this way naturally come as a $\CC^\times$-family of holomorphic symplectic forms $\varpi(\zeta)$ satisfying the above assumptions. Moreover, the Gaiotto--Moore--Neitzke formalism gives a construction of $\kappa$. Our concrete version of the twistor theorem is very convenient since it doesn't require the construction of $\kappa$. 

Readers are referred to \cite[\S3F]{hitchin:hkSUSY}.  
In places, our presentation closely follows Andy Neitzke's course notes\cite{neitzke:higgsNotes}.

\bigskip

Readers familiar with hyper-K\"ahler structures should skip directly to Proposition \ref{prop:varpiRealJ} and the following.

\subsubsection*{Acknowledgements.} 
We thank
J. Fine, R. Mazzeo and A. Neitzke for helpful conversations.
We are also grateful to R. Mazzeo for helpful comments on a draft of this work. LF is partially supported by NSF grant DMS-2005258.
This material is based upon work supported by the National Science
Foundation under Grant No. DMS-1928930, while LF was in
residence at the Simons Laufer Mathematical Sciences Institute
(formerly MSRI) in Berkeley, California, during the Fall 2024
semester.

\section{Background}

In this first section, we establish our conventions and explain how a pseudo-hyper-K\"ahler structure gives a family of holomorphic symplectic structures. This section contains no new results though we clarify some aspects of the proof in \cite{hitchin:hkSUSY}  regarding holomorphic vs antiholomorphic subbundles (see the footnotes). The culmination is Lemma \ref{lem:realHol} in which we give a concrete identification between $\M$ and the real manifold that the twistor theorem produces from its twistor space.

A pseudo-hyper-K\"ahler\footnote{We allow for signatures which are not positive definite not because we have any special interest in them here, but because, as has been observed before (e.g. \cite{BielawskiFoscolo}) contrary to the statement of Theorem 3.3 in \cite{hitchin:hkSUSY}, the twistor formalism 
only guarantees the existence of a pseudo-hyper-K\"ahler structure. Note that for a pseudo-hyper-K\"ahler manifold, the signature will be of the form $(p, q)$ where $p, q$ are multiples of four.} manifold is a generalization of a hyper-K\"ahler manifold where we remove the condition that the non-degenerate metric $g$ be positive-definite. 

\begin{figure}[h!]
    \begin{centering} 
    \includegraphics[height=1.5in]{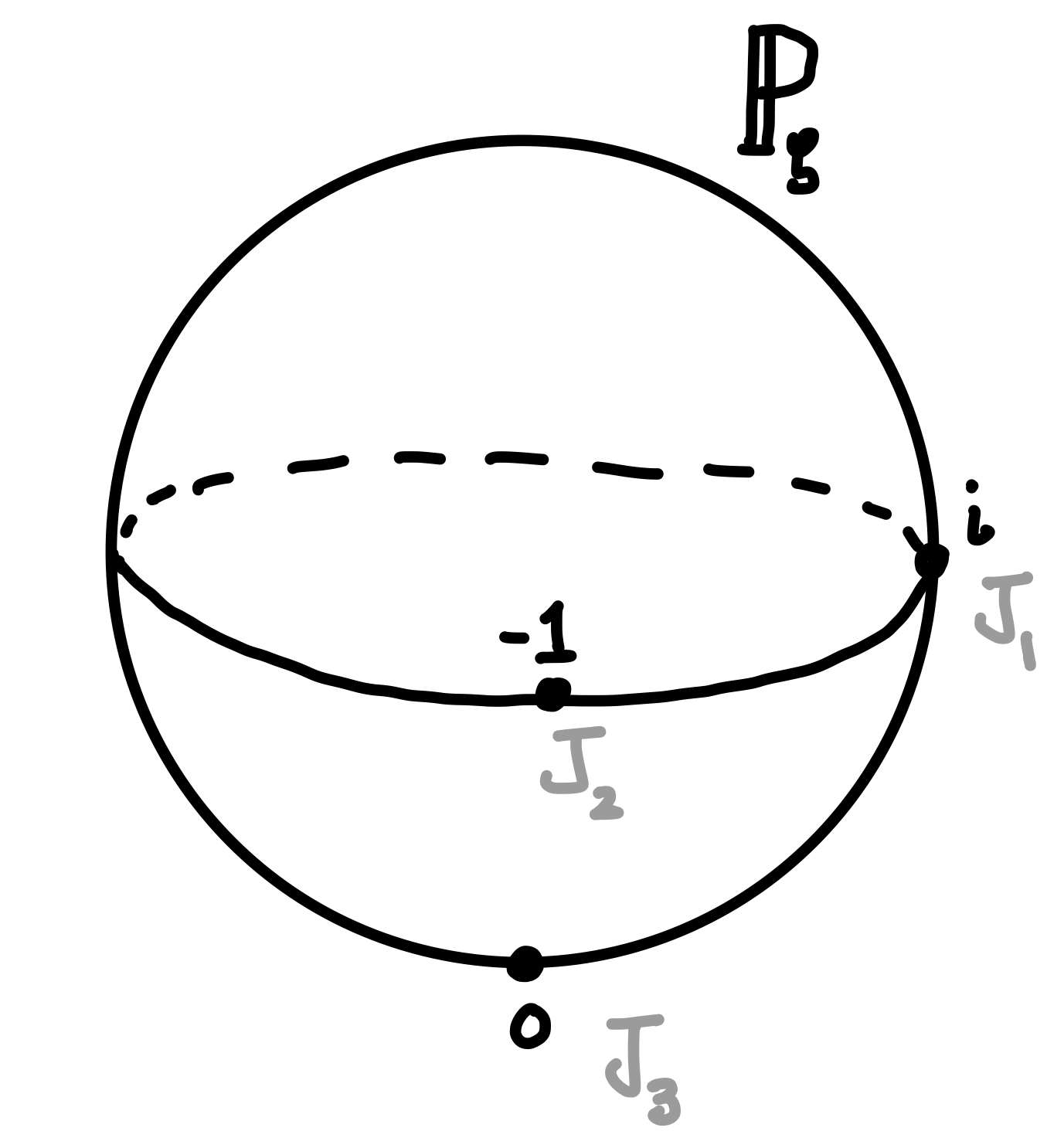} 
    \caption{\label{fig:diagram} 
    The twistor sphere.
    }
    \end{centering}
    \end{figure}

\begin{defn}
A \emph{pseudo-hyper-K\"ahler manifold} $\mathcal{M}$ is a pseudo-Riemannian manifold $(\mathcal{M}, g)$ (i.e. a manifold $\mathcal{M}$ equipped with a non-degenerate symmetric two-tensor $g$ which need not be positive-definite) equipped with a $\mathbb{P}^1$-worth of symplectic and complex structures
\be \omega^{(\zeta)} = \sum_{\alpha=1}^3 c_\alpha^{(\zeta)} \omega_\alpha \ , \quad J^{(\zeta)} = \sum_{\alpha=1}^3 c_\alpha^{(\zeta)} J_\alpha  \label{eq:zetaKahler} \ee
determined by a triplet of symplectic structures $(\omega_1, \omega_2, \omega_3)$ and a triplet of complex structures $(J_1, J_2, J_3)$
satisfying the unit quaternion algebra relations
\be J_\alpha J_\beta = \epsilon_{\alpha\beta\gamma} J_\gamma - \delta_{\alpha\beta} \ , \label{eq:quat} \ee
where $\epsilon_{\alpha\beta\gamma}$ is the totally antisymmetric tensor with $\epsilon_{123}=1$ and $\delta_{\alpha\beta}$ is the Kronecker delta.
Here,  $c^{(\zeta)}$ is a unit vector in $\RR^3$ which is related to $\zeta\in \PP^1$ via stereographic projection:
\be c^{(\zeta)} = \frac{1}{1+|\zeta|^2}(2\Imag \zeta, -2\Real \zeta, 1-|\zeta|^2) \ , \quad \zeta = \frac{i c_1^{(\zeta)}-c_2^{(\zeta)}}{1+c_3^{(\zeta)}} \ . \label{eq:stereo} \ee
We additionally impose that for $\alpha=1, 2, 3$, $(g, \omega_\alpha, J_\alpha)$ satisfy the (usual) pseudo-K\"ahler compatibility condition, which we state in three ways:
\be \omega_\alpha(v_1, v_2) = g(J_\alpha v_1, v_2) \ , \quad g(v_1,v_2) = \omega_\alpha(v_1, J_\alpha v_2) \ , \quad J_\alpha = - g^{-1} \omega_\alpha \ . \ee
In the last equation, we regard $g$ and $\omega_\alpha$ as isomorphisms from $T\M$ to $T^* \M$, so that $J_\alpha$ is an automorphism of $T\M$.
\end{defn}

All of the above data may be recovered from the three pseudo-K\"ahler forms $\omega_\alpha$: for example, $J_\alpha^2=-1$ (no sum on $\alpha$) implies that $J_\alpha = \omega_\alpha^{-1}g$, and so we can recover $J_3$ and $g$ via $J_3 = J_1 J_2 = -\omega_1^{-1}\omega_2$ and $g = \omega_3 J_3 = - \omega_3 \omega_1^{-1} \omega_2$.

Note that the antipodal map $\zeta\mapsto -1/\bar\zeta$ negates the vector $c^{(\zeta)}$, and thus takes $J^{(\zeta)}$ to the conjugate complex structure $-J^{(\zeta)}$. 

\bigskip

We will refer to complex structures as orthogonal if their corresponding unit vectors \eqref{eq:stereo} are orthogonal in $\RR^3$.  Looking forward, we will make use of the following Lemma in the proof of Lemma \ref{prop:rotate}:

\begin{lemma}The unit quaternion algebra relations in \eqref{eq:quat} are preserved by the natural $SO(3)$ action on $S^2$ (which acts on $\zeta$ via $PSU(2)$ fractional linear transformations).\end{lemma}\begin{proof} Explicitly, the $SU(2)$ adjoint action
\be \twoMatrix{ic_3^{(\zeta)}}{c_1^{(\zeta)}+i c_2^{(\zeta)}}{-c_1^{(\zeta)}+ic_2^{(\zeta)}}{-ic_3^{(\zeta)}} \mapsto \twoMatrix{u}{v}{-\bar v}{\bar u} \twoMatrix{ic_3^{(\zeta)}}{c_1^{(\zeta)}+ic_2^{(\zeta)}}{-c_1^{(\zeta)}+ic_2^{(\zeta)}}{-i c_3^{(\zeta)}} \twoMatrix{\bar u}{-v}{\bar v}{u} \ , \label{eq:psu2} \ee
where $|u|^2+|v|^2=1$ (and negating both $u$ and $v$ has no effect), corresponds to
\be \zeta \mapsto \frac{u\zeta+v}{-\bar v \zeta + \bar u} \ . \label{eq:flPSU2} \ee
In particular, we note that the $U(1)$ subgroup with $v=0$ rotates $J_1$ and $J_2$ into each other while fixing $J_3$. 
\end{proof}

\bigskip

If we want to refer to $\M$ as a complex manifold with the complex structure $J^{(\zeta)}$, we will use the notation $\M^{(\zeta)}$. Using this, we introduce the following complex vector bundles over the real manifold $\M \times \PP^1$: $T^{(1,0)}\M\to \M\times \PP^1$, $T^{(0,1)}\M\to \M\times \PP^1$, $T^*_{(1,0)}\M\to \M\times \PP^1$, and $T^*_{(0,1)}\M\to \M\times \PP^1$ whose restrictions to $\M\times\{\zeta\}$ are, respectively, the holomorphic tangent bundle $T^{(1,0)}\M^{(\zeta)}$, the anti-holomorphic tangent bundle $T^{(0,1)}\M^{(\zeta)}$, the holomorphic cotangent bundle $T^*_{(1,0)}\M^{(\zeta)}$, and the anti-holomorphic cotangent bundle $T^*_{(0,1)}\M^{(\zeta)}$. The fibers are related as follows:
\begin{proposition}
\label{prop:iso0tozeta}\hfill\begin{enumerate}[(a)]\item 
The map $1+\zeta J_1$ is an isomorphism from $T^{(0,1)} \M^{(0)}$ to $T^{(0,1)} \M^{(\zeta)}$. 
\item Similarly, the map
$1-\zeta J_1'$ is an isomorphism from $T^*_{(1,0)} \M^{(0)}$ to $T^*_{(1,0)} \M^{(\zeta)}$,
where primes here indicate transposes --- e.g., $(J_1' \sigma)(v)=\sigma(J_1 v)$.\end{enumerate}
\end{proposition}
\begin{proof}
We first observe that $(-J_1', -J_2', -J_3')$ satisfy \eqref{eq:quat}.

Using \eqref{eq:stereo}, one can show that for all $\zeta\in \CC$,
\be J^{(\zeta)} (1+\zeta J_1)|_{T^{(0,1)} \M^{(0)}} = -i (1+\zeta J_1)|_{T^{(0,1)} \M^{(0)}} \ , \label{eq:holo1} \ee
i.e., $(1+\zeta J_1)$ maps $T^{(0,1)} \mathcal{M}^{(0)}$ to $T^{(0,1)} \mathcal{M}^{(\zeta)}$, the $-i$ eigenspace of $J^{(\zeta)}$;
and similarly,
\be (J^{(\zeta)})' (1-\zeta J_1')|_{T^*_{(1,0)} \M^{(0)}} = i (1-\zeta J_1')|_{T^*_{(1,0)} \M^{(0)}} \ . \label{eq:holo2} \ee
 Since $J_1$ is an isomorphism from $T^{(0,1)} \M^{(0)}$ to $T^{(1,0)}\M^{(0)}$, we find that $1+\zeta J_1$ is an injective map from $T^{(0,1)} \M^{(0)}$ to $T^{(0,1)} \M^{(\zeta)}$, and hence is an isomorphism. (Indeed, the inverse is \begin{equation}\label{eq:inverse}(1 + \zeta J_1)^{-1}:=\piecewise{\frac{1}{1+\zeta^2}(1-\zeta J_1)}{\zeta\not=\pm i}{\frac{1}{1-\zeta^2}(1+i\zeta J_2)}{\rm else}.\end{equation} The expression used in the $\zeta=\pm i$ case in fact works for all $\zeta\not=\pm 1$, as follows from writing $(1+\zeta J_1)|_{T^{(0,1)}\M^{(0)}} = (1-i\zeta J_2)|_{T^{(0,1)}\M^{(0)}}$). 
\end{proof}

\begin{corollary}[Corollary to Proposition \ref{prop:iso0tozeta}] \label{cor:holbundle} Let $(T^{(0,1)} \M^{(0)})_m$ be the  trivial bundle over $\PP^1$.
Then for each $m \in \M$,
\begin{enumerate}[(a)]
\item $(T^{(0,1)}\M)_m \cong (T^{(0,1)} \M^{(0)})_m \otimes \Oo(-1)$
\item $(T^*_{(1,0)}\M)_m\cong (T^*_{(1,0)}\M^{(0)})_m \otimes \Oo(-1)$
\end{enumerate}
\end{corollary}
We can use the natural pairing of sections $$\Gamma((T^{(1,0)} \M)_m) \times  \Gamma((T^*_{(1,0)} \M)_m) \to \Gamma(\CC \times \PP^1 \to \PP^1)$$ to equip $(T^{(1,0)} \M)_m$ with the structure of a holomorphic bundle.
Namely, \begin{definition}\label{def:holrealdef1}A smooth section $s$ of $(T^{(1,0)} \M)_m$ 
is \emph{holomorphic}
if for all holomorphic sections $\sigma$ of $(T^*_{(1,0)} \M)_m$, $\sigma(s)$ is a holomorphic section of $\mathcal{O} \to \PP^1$, i.e. a constant.
\end{definition}

\begin{corollary}[Corollary to Corollary \ref{cor:holbundle}]\label{cor:duality}
Let $(T^{(0,1)} \M^{(0)})_m$ be the  trivial bundle over $\PP^1$.
Then for each $m \in \M$,
\begin{enumerate}[(a)]
\item $(T^{(1,0)}\M)_m \cong (T^{(1,0)} \M^{(0)})_m \otimes \Oo(1)$
\item $(T^*_{(0,1)} \M)_m \cong (T^*_{(0,1)}\M^{(0)})_m \otimes\Oo(1)$.\footnote{\label{foot:holvsantihol}Repeating (or conjugating the results of) the computations in \eqref{eq:holo1} and \eqref{eq:holo2}, one finds that these two bundles have natural \emph{anti-holomorphic} structures. However, we can straightforwardly relate these anti-holomorphic subbundle structures to the holomorphic bundle structures described in the main text. For example, since $(T^{(0,1)}\M)_m$ complements $(T^{(1,0)}\M)_m$ in $(T_\CC \M)_m\times \PP^1$, $(T^{(1,0)}\M)_m \cong (T_\CC \M)_m\times \PP^1/(T^{(0,1)}\M)_m$ gives $(T^{(1,0)}\M)_m$ the structure of a holomorphic vector bundle (this is equivalent to the construction via duality in the text, since $(0,1)$-vectors are precisely the vectors which are annihilated by all $(1,0)$-forms), but on the other hand the non-degeneracy of the pseudo-Riemannian metric allows us to characterize $(T^{(1,0)}\M)_m$ as the sub-bundle of $(T_\CC\M)_m$ which is orthogonal to all of $(T^{(0,1)}\M)_m$. If $v$ is a holomorphic local section of $(T^{(0,1)}\M)_m$ then the condition $g(\bar v, w)=0$ for $w\in (T_\CC\M)_m$ to be orthogonal to $v$ is anti-holomorphic in $\zeta$, and so this gives $(T^{(1,0)}\M)_m$ a natural anti-holomorphic subbundle structure, which---using the isomorphism of complex bundles $T^{(1,0)} \M^{(0)} \simeq T^{(0,1)} \M^{(\infty)}$---we denote $(T^{(1,0)}\M)_m \simeq (T^{(0,1)} \M^{(\infty)})_m \otimes \overline{\mathcal{O}(-1)}$. It is worth noting, as a sanity check, that an anti-holomorphic $\Oo(-1)$ bundle and a holomorphic $\Oo(1)$ bundle are isomorphic as complex line bundles.}
\end{enumerate}
\end{corollary}

\begin{definition}\label{def:twistorspace}
Given a hyper-K\"ahler manifold $\M$, the \emph{twistor space} $\Z$ of $\M$ is the total space of
the holomorphic fibration $$p: \Z\to \PP^1$$ with fiber $\M^{(\zeta)}$, equipped with the inherited almost complex structure. 
\end{definition}
As a real manifold, $\Z$ is the product $\M\times \mathbb{P}^1$. \begin{proposition} \label{prop:twistoriscomplex}The twistor space $\Z$ is a complex manifold.\end{proposition}
We give further properties of the twistor space $\Z$ in  Proposition \ref{prop:twistorprop} and Proposition \ref{prop:normal}.
\begin{proof}
To see that the almost complex structure on $\Z$ is integrable, we verify the criterion that $(dw)_{(0,2)}=0$ for all $(1,0)$-forms $w$.\footnote{Recall that an almost complex structure gives a decomposition of an $n$-form into $(p,q)$-forms for $p+q =n$. The operators $\partial$ and $\overline{\partial}$ are defined by 
\[ \partial|_{\Omega^{(p,q)}} = \pi_{(p + 1, q)} \circ \de |_{\Omega^{p,q}} \ , \qquad \overline{\partial}|_{\Omega^{(p,q)}} = \pi_{(p, q+1)} \circ \de |_{\Omega^{p,q}} .\]
An almost complex structure is integrable if, and only if, $\de = \partial + \overline{\partial}$.} This is clear if $w=d\zeta$, so we focus on $(1,0)$-forms of the form $w=(1-\zeta J_1')v$, where $v\in T^*_{(1,0)}\M^{(0)}$. (This is in a patch containing 0 but not $\infty$; the analogous results in a patch containing $\infty$ are similar.) We note that $dw=d_\M(1-\zeta J_1') v - d\zeta \wedge J_1' v$, where $d_\M$ is the exterior derivative only on $\M$. Because the Levi-Civita connection on $\M$ is torsion-free, we may equivalently regard $d_\M$ as the exterior covariant derivative $d_{\M,\nabla}$: $dw=dx^i\wedge \nabla_{\partial_{x^i}} (1-\zeta J_1') v - d\zeta\wedge J_1' v$, where $x^i$ is a choice of coordinates on $\M$.\footnote{There is no reason here that the Levi-Civita connection should be preferred --- the present argument works so long as there exists any torsion-free connection with respect to which $J_1$, $J_2$, and $J_3$ are parallel. This observation leads to the definition of a \emph{hypercomplex} manifold, which is a quaternionic analogue of a complex manifold for which there always exists a unique such connection \cite[Theorems 10.3 and 10.5]{obata:conn}.} Since $(J^{(\zeta)})'$ is covariantly constant, $(J^{(\zeta)})' \nabla_{\partial_{x^i}} (1-\zeta J_1') v = i \nabla_{\partial_{x^i}} (1-\zeta J_1') v$, and so $dw$ has no $(0,2)$ component, as claimed.\footnote{\label{footnote:correction}
We mention a pair of signs errors in  \cite{hitchin:hkSUSY} that cancelled since they were subtle and confused us. In \S3F therein (see the paragraph beginning with (3.80)), the authors state that $(T^{(1,0)}\M)_m$ is a holomorphic subbundle (as opposed to quotient bundle) of the trivial holomorphic bundle $(T_\CC\M)_m\times \PP^1$, essentially via a computation analogous to \eqref{eq:holo1}. They then claim that this subbundle is holomorphically a direct sum of $\Oo(1)$ bundles. This is not possible: $\Oo(1)$ cannot be a subbundle of a trivial bundle. The bundle they describe is actually a direct sum of $\Oo(-1)$ bundles, which as we have seen is the expected structure of $(T^{(0,1)}\M)_m$, not $(T^{(1,0)}\M)_m$. The reason for this discrepancy is that their orientation of the twistor sphere is opposite from that of \eqref{eq:stereo} (see their (3.70)). It is necessary to employ our orientation in order for the almost complex structure on $\Z$ to be integrable. The proof of this integrability which we have copied from \cite{hitchin:hkSUSY} uses \eqref{eq:holo2}; \cite{hitchin:hkSUSY} arrives at \eqref{eq:holo2} in spite of the conjugation of $\PP^1$ by assuming (see (3.74)) that the complex structures $(J_1',J_2',J_3')$ on $T^* \M$ satisfy \eqref{eq:quat}; however, it is actually $(-J_1',-J_2',-J_3')$ which satisfy \eqref{eq:quat}.} 
\end{proof}

Some of the holomorphic vector bundles we defined above over $\PP^1$ extend naturally to all of $\Z$: $T^{(1,0)} \M = \ker (p_*: T^{(1,0)}\Z\to T^{(1,0)}\PP^1)$ and $T^*_{(1,0)} \M \cong T^*_{(1,0)}\Z/p^* T^*_{(1,0)}\PP^1$. 

\bigskip

Finally, we turn to a powerful means of encoding the pseudo-hyper-K\"ahler structure. Let $V={\rm span}_\CC\{\omega_1,\omega_2,\omega_3\}$ be a subbundle of $\wedge^2 T_\CC\M\times \PP^1 \to \Z$. 
\begin{lemma}[c.f. {\cite[p.577]{hitchin:hkSUSY}}] 
\label{lem:L}Then, $L=V\cap \wedge^2 T^{(1,0)}\M$ is a complex line bundle over $\Z$, spanned by 
\be \varpi(\zeta) = - \frac{i}{2\zeta} \omega_+ + \omega_3 - \frac{i}{2} \zeta \omega_- ,  \ee
where $\omega_\pm = \omega_1 \pm i \omega_2$,
for $\zeta \in \CC^\times$; $\omega_+$ for $\zeta=0$; and $\omega_-$ for $\zeta=\infty$.

Moreover, the complex line bundle has the structure of a holomorphic line bundle bundle: restricting to a point $m\in \M$, $L_m \to \PP^1$ is isomorphic to $\Oo(-2)$; said another way, $\varpi(\zeta)$ is a holomorphic section of $\wedge^2 T^*_{(1,0)}\M \otimes p^* \Oo(2) \to \Z$. 
\end{lemma}\begin{proof}By the symmetry between the various complex structures, it suffices to compute the dimension of this bundle over $\M\times\{0\}$. When $\zeta=0$, the pseudo-K\"ahler form $\omega_3$ is a $(1,1)$-form, $\omega_+$ is a $(2,0)$-form, and $\omega_-$ is a $(0,2)$-form, where $\omega_\pm = \omega_1 \pm i \omega_2$. For, we have
\begin{align}
\omega_\pm(J_3 v, w) &= \omega_1(J_1 J_2 v, w) \mp i \omega_2(J_2 J_1 v, w) = -g(J_2 v, w) \pm i g(J_1 v, w) \nonumber \\
&= -\omega_2(v,w) \pm i \omega_1(v, w) = \pm i \omega_\pm(v,w) \ .
\end{align}
More generally, when $\zeta\in \CC^\times$, a similar computation demonstrates that $L$ is spanned by $\varpi(\zeta)$.
(And, when $\zeta=\infty$, $L$ is spanned by $\omega_-$.) Alternatively, one may show that this is a $(2,0)$-form on $\M^{(\zeta)}$ by proving that for all vectors $v,w$ we have
\be \varpi(\zeta)(v,w) = -\frac{i}{2\zeta} \omega_+((1-\zeta J_1)v,(1-\zeta J_1)w) \ee
and then using the fact that $1-\zeta J_1$ (from \eqref{eq:inverse}) is a map from $T^{(0,1)}\M^{(\zeta)}$ to $T^{(0,1)}\M^{(0)}$. 

The holomorphic line bundle structure on $L$ follows. Note that $\varpi(\zeta)$ is annihilated by the Dolbeault operator on $\M^{(\zeta)}$ because it is closed.
\end{proof}

Not only is $\varpi(\zeta)$ closed, but it is a holomorphic symplectic form on $\M^{(\zeta)}$.
To see this, we first claim the following:
\begin{lemma} \label{prop:rotate}
For every $\zeta\in \widetilde{\PP^1}$, the real blowup of $\PP^1$ at $0,\infty$, there are unique complex structures $K^{(\zeta)}$ and $I^{(\zeta)}$ such that $(K^{(\zeta)},I^{(\zeta)},J^{(\zeta)})$ are related to $(J_1,J_2,J_3)$ by an element of $SO(3)$ and
\be (|\zeta|+|\zeta|^{-1})^{-1}\varpi(\zeta) = \half(\omega_{K^{(\zeta)}} + i \omega_{I^{(\zeta)}}) \ , \label{eq:rescaledVarpi} \ee
where $\omega_{K^{(\zeta)}}$ and $\omega_{I^{(\zeta)}}$ are the associated pseudo-K\"ahler forms.
\end{lemma}
\begin{corollary} \label{cor:holsymp}The 2-form
$\varpi(\zeta)$ is a holomorphic symplectic form on $\M^{(\zeta)}$ --- i.e., it is non-degenerate, in the sense that it defines an isomorphism $T^{(1,0)}\M^{(\zeta)} \to T^*_{(1,0)} \M^{(\zeta)}$. 
\end{corollary}
\begin{proof}[Proof of Corollary \ref{cor:holsymp}] Assuming Lemma \ref{prop:rotate} for the moment, we can give two proofs for every $\zeta\in \widetilde{\PP^1}$ that $\varpi(\zeta)$ in \eqref{eq:rescaledVarpi} is non-degenerate.
We show for every non-zero $v\in (T^{(1,0)}\M^{(\zeta)})_m$ there is a $w\in (T^{(1,0)}\M^{(\zeta)})_m$ such that $\varpi(\zeta)(v,w)\not=0$. 

First,
\begin{align} (|\zeta|+|\zeta|^{-1})^{-1}\varpi(\zeta)(v,\cdot) = 0 &\Rightarrow (|\zeta|+|\zeta|^{-1})^{-1}\varpi(\zeta)(v+\bar v,\cdot) = 0 \nonumber\\
&\Rightarrow \omega_{K^{(\zeta)}}(v+\bar v,\cdot) = 0 \Rightarrow v+\bar v = 0 \Rightarrow v=0 \ .
\end{align}

Second, if $\bar x\in (T^{(0,1)}\M^{(\zeta)})_m$ then $I^{(\zeta)} \bar x=i I^{(\zeta)} J^{(\zeta)} \bar x = i K^{(\zeta)} \bar x$, and so $g(v,\bar x) = \omega_{K^{(\zeta)}}(v, K^{(\zeta)}\bar x) = \omega_{I^{(\zeta)}}(v, I^{(\zeta)} \bar x) = i \omega_{I^{(\zeta)}}(v, K^{(\zeta)} \bar x)$ and
\be (|\zeta|+|\zeta|^{-1})^{-1} \varpi(\zeta)(v, K^{(\zeta)} \bar x) = g(v,\bar x) \ . \label{eq:holoG} \ee
This not only reduces the non-degeneracy of \eqref{eq:rescaledVarpi} to that of $g$, but also shows that if $\M$ is Riemannian---as opposed to pseudo-Riemannian---then we may take $w=K^{(\zeta)} \bar v$. 
\end{proof}
\begin{proof}[Proof of Lemma \ref{prop:rotate}]
We consider choices of complex structures $K^{(\zeta)}$ and $I^{(\zeta)}$ such that $\begin{pmatrix} K^{(\zeta)} \\ I^{(\zeta)} \\ J^{(\zeta)} \end{pmatrix} = R \begin{pmatrix} J_1 \\ J_2 \\ J_3 \end{pmatrix}$ for some $R\in SO(3)$. Changing bases, we have
\be c_1 J_1 + c_2 J_2 + c_3 J_3 = c'_1 K^{(\zeta)} + c'_2 I^{(\zeta)} + c'_3 J^{(\zeta)} \ , \ee
so that $c'=Rc$. This makes it clear that the allowed $R$ are those which map $\zeta$ to 0. In the parametrization of \eqref{eq:psu2} and \eqref{eq:flPSU2}, we thus have $v=-u\zeta$. The condition $|u|^2+|v|^2=1$ then allows us to write $u=\frac{e^{i\theta}}{\sqrt{1+|\zeta|^2}}$ and $v=-\frac{\zeta e^{i\theta}}{\sqrt{1+|\zeta|^2}}$ for some $\theta\in\RR$; this phase freedom is associated to the freedom of rotating $K^{(\zeta)}$ and $I^{(\zeta)}$ into each other while fixing $J^{(\zeta)}$. Since the vector of pseudo-K\"ahler forms also transforms via $R$ (so that the dot product $c\cdot \omega$ is invariant), we have $\begin{pmatrix} \omega_{K^{(\zeta)}} \\ \omega_{I^{(\zeta)}} \\ \omega_{J^{(\zeta)}} \end{pmatrix} = R \begin{pmatrix} \omega_1\\ \omega_2\\ \omega_3 \end{pmatrix}$. The top right entry of \eqref{eq:psu2} then demonstrates that
\be \omega_{K^{(\zeta)}}+i\omega_{I^{(\zeta)}} = \frac{2i\zeta}{1+|\zeta|^2} e^{2i\theta} \varpi(\zeta) \ . \ee
So, we have \eqref{eq:rescaledVarpi} for a unique choice of $\theta\in \RR/\pi\ZZ$ --- i.e., for a unique $R$.
\end{proof}

\begin{proposition}[\cite{hitchin:hkSUSY}] \label{prop:twistorprop}
Suppose $\M$ is a pseudo-hyper-K\"ahler manifold with twistor space $\Z$. Then, 
\begin{enumerate}[(1)]
    \item $\Z$ is a complex manifold, with holomorphic projection $p: \Z \to \PP^1$;
    \item $\Z$ carries a twisted fiberwise holomorphic symplectic form $\varpi(\zeta)$, which is a holomorphic section of $\wedge^2 T^*_{(1,0)} \M \otimes p^*\mathcal{O}(2) \to \Z$;
    \item $\Z$ carries a real structure $\rho: \Z \to \Z$, i.e. an antiholomorphic involution, such that
    \begin{enumerate}[(a)]\item 
    $\rho$ covers the antipodal involution \begin{align*}
        \sigma: \PP^1 & \to \PP^1\\
        \zeta &\mapsto -1/ \overline{\zeta}
    \end{align*}
    \item $\rho^* \varpi = \overline{\varpi}$
\end{enumerate}
\end{enumerate}
\end{proposition}
\begin{proof}[Remark on Proof]
Proposition \ref{prop:twistoriscomplex} gives (1), while Lemma \ref{lem:L} and Corollary \ref{cor:holsymp} gives (2). For (3), using the identification of $\Z$ with $\M \times \PP^1$ as real manifolds, take $\rho(m, \zeta)=(m, - 1/\overline{\zeta}).$ The statement $\rho^*\varpi = \overline{\varpi}$ amounts to the fact that if we replace $\zeta$ with $-1/\overline{\zeta}$ in $\varpi(\zeta)$, then we get $-\overline{\varpi(\zeta)}$, for $\varpi(\zeta)$ in Lemma \ref{lem:L}.
\end{proof}

\begin{defn}[Real section of $\Z \to \PP^1$]A smooth section $s: \PP^1 \to \Z$ of $p: \Z \to \PP^1$ is called \emph{real} 
if $\rho(s(\zeta))=s(\sigma(\zeta))$ for all $\zeta \in \PP^1$.
\end{defn}

\begin{proposition}[{\cite[Theorem 3.3ii]{hitchin:hkSUSY}}] \label{prop:normal}Let $\M^{4r}$ be a pseudo-hyper-K\"ahler manifold, and let $\Z$ be its twistor space . For each $m \in \M$, there is a  corresponding real holomorphic section of $p: \Z \to \PP^1$ given by 
\begin{align*}
    s_m: \PP^1 &\to \Z \simeq \M \times \mathbb{P}^1\\
    \zeta &\mapsto   \qquad (m, \zeta)
\end{align*}
corresponding to the points $m \in \M$. 
The normal bundle to $s_m(\mathbb{P}^1) \subset \Z$ is isomorphic to $\mathcal{O}(1)^{\oplus 2r}$.
\end{proposition}
\begin{proof}[Remark on Proof]
The key observation here is that $(T^{(1,0)} \M)_m \simeq (T^{(1,0)} \M^{(0)})_m  \otimes \mathcal{O}(1)$.
\end{proof}

\bigskip
Lastly, we note that the tangent space to $s_m$ within the space of real holomorphic sections of $(T^{(1,0)} \M)_m,$ which we denote
$H^0_\RR((T^{(1,0)}\M)_m),$ can naturally be identified with the tangent space $T_m\M$ via an intermediate identification with $(T^{(1,0)} \M^{(0)})_m$ in Lemma \ref{lem:realHol}. After Theorem \ref{thm:HKLR}, we will discuss how this is the key observation to show that when Hitchin--Karlhede--Lindstr\"om--Ro\v{c}ek's twistor theorem is applied to a twistor space $\mathcal{Z}$ obtained from a pseudo-hyper-K\"ahler manifold $\M$, it recovers the original pseudo-hyper-K\"ahler structure.  

\begin{definition}\label{def:holrealdef}
A smooth section $s$ of $(T^{(1,0)} \M)_m \subset (T_{\CC} \M)_m \subset T_\CC \Z$ is \emph{real} if under 
the induced conjugate linear differential map\footnote{Concretely, the trivialization as real manifolds $\mathcal{Z} \simeq \M \times \PP^1$ induces a trivialization of $T\M \to \PP^1$ and of $T_\CC \M \to \PP^1$. Given $(v, \lambda) \in (T\M)_m \otimes \CC \subset (T_\CC\Z)_{(m, \zeta)} $, the image under $\de \rho$ is $(v, \overline{\lambda}) \in (T \M)_m \otimes \CC \subset (T_\CC \Z)_{(m, -1/\overline{\zeta})}$.}
$\rho_*: T_\CC \Z \to T_\CC \Z$ covering the antipodal involution $\sigma$, one has
$$\rho_* \circ s = s \circ \sigma.$$
\noindent\underline{Notation:} We will use the complex conjugate notation $\overline{\cdot}$ to denote the real structure $ \rho_*$ on $(T_\CC \M)_m$, i.e. $\overline{s(\zeta)} = \rho_*(s(\zeta))$.
\end{definition}

We stress that the reality condition is defined by regarding $(T^{(1,0)}\M)_m$ as a (complex or antiholomorphic) subbundle of  the trivial complex bundle $(T_\CC\M)_m$ whereas the holomorphicity condition is defined in Definition \ref{def:holrealdef} via duality with $(T^*_{(1,0)}\M)_m$  
(see extended discussion in Footnote \ref{foot:holvsantihol}), and the translation between these two languages is not immediate.

\begin{lemma} \label{lem:realHol}\footnote{Comparing Lemma \ref{lem:realHol} to \cite[(3.102)]{hitchin:hkSUSY}, we take the opportunity to identify $j=J_1 \overline{\cdot}$.}
Suppose $\Z$ is the twistor space of a pseudo-hyper-K\"ahler manifold $\M$.
Represent a holomorphic section of $(T^{(1,0)} \M)_m \to \PP^1$ under the holomorphic identification $(T^{(1,0)} \M)_m \simeq (T^{(1,0)} \M^{(0)})_m  \otimes \mathcal{O}(1)$ as a section $\zeta \mapsto v(\zeta)$ using the standard trivialization\footnote{The line bundle $\mathcal{O}(k) \to \PP^1$ can be built by gluing the trivial bundle $\CC_\zeta \times \CC \to \CC_\zeta = \PP^1-\{0\}$ with the trivial bundle $\CC_{\widetilde{\zeta}} \times \CC \to \CC_{\widetilde{\zeta}} =\PP^1-\{\infty\}$ via the map
$$(\zeta, f(\zeta)) \mapsto (\widetilde{\zeta}=\zeta^{-1}, \widetilde{\zeta}^k f(\widetilde{\zeta}^{-1})).$$
Let $e(\zeta)=(\zeta,1)$ be the canonical section of the trivial bundle of $\CC_\zeta \times \CC \to \CC_\zeta$. Thus, holomorphic section of $\mathcal{O}(k)$ in the trivialization by $e$ are given by degree $k$ polynomials $f$; holomorphic sections of $\mathcal{O}(1)$ are given by linear polynomials $f$.} 
$e$ of $\mathcal{O}(1)|_\CC$.
The map
\begin{align*}
    (T^{(1,0)} \M^{(0)})_m &\to H^0_{\RR}((T^{(1,0)}\M)_m)\\
    a &\mapsto v(\zeta) =a - \zeta J_1 \overline{a}
\end{align*}
is an isomorphism.
\end{lemma}
\begin{proof}[Proof of \ref{lem:realHol}] We will characterize sections $s$ of the trivial bundle $T_\CC \M \to \PP^1$ such that (1) $s$ lies in the subbundle $T^{(1,0)} \M \to \PP^1$, (2) $s$ is real, and (3) $s$ is holomorphic. The first condition is easiest to state in terms of a related section $v$ which we will introduce, the second condition is easiest to state in terms of $s$ itself, and the third is easiest to state in terms of a a related section $t$ that we will introduce. All of these sections appear in Figure \ref{fig:twistortriv}.
\begin{figure}[h!]
\begin{centering} 
\includegraphics[height=2.8in]{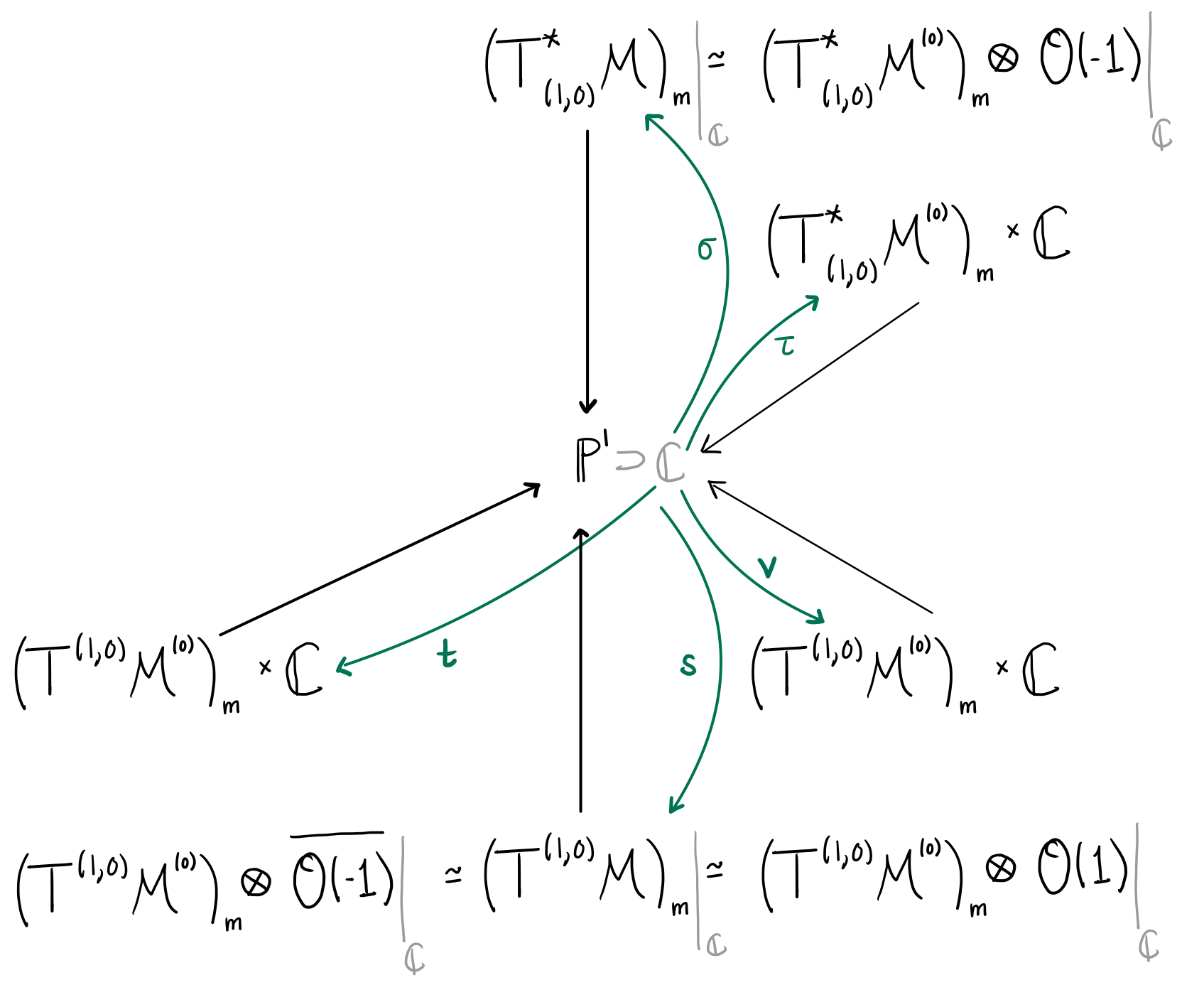} 
\caption{\label{fig:twistortriv} Relevant sections in the Proof of Lemma \ref{lem:realHol}.}
\end{centering}
\end{figure}

The trivial bundle $T_\CC \M \to \PP^1$ contains the subbundle $T^{(1,0)} \M \to \PP^1$ as a subbundle. 
As a subbundle of $(T_\CC \M)_m$, $(T^{(1,0)} \M)_m$ is naturally an anti-holomorphic $\mathcal{O}(-1)$ bundle, which we may denote $(T^{(1,0)} \M)_m \simeq T^{(0,1)} \M^{(\infty)} \otimes \overline{\mathcal{O}(-1)}$, via the identification of 
$T^{(1,0)} \M^{(\zeta)} = (1+ \overline{\zeta} J_1) T^{(1,0)} \M^{(0)}$, dual to the identification in Proposition \ref{prop:iso0tozeta}; consequently, a smooth section $s$ of $(T^{(1,0)} \M)_m \subset (T_\CC \M)_m$ over $\CC$ is related to a smooth section $t$ of the trivial bundle $T^{(1,0)} \M^{(0)} \times \CC \to \CC$ by 
\[s=(1+ \overline{\zeta} J_1)t.\]

In terms of $s$, a real section is one that satisfies 
$$s(\zeta) = \overline{s(-1/\zeta)}.$$

\bigskip

At the same time, $(T^{(1,0)} \M)_m \simeq (T^{(0,1)} \M^{(0)})_m \otimes \mathcal{O}(1)$ as the quotient bundle  $(T^{(1,0)} \M)_m = (T_\CC \M)_m/(T^{(0,1)}\M)_m $; here, as a holomorphic bundle, $(T_\CC \M)_m \simeq (T^{(1,0)} \M^{(0)})_m \otimes \mathcal{O}^{\oplus 2}$, meanwhile $(T^{(0,1)} \M)_m\simeq (T^{(0,1)} \M^{(0)})_m \otimes \mathcal{O}(-1)$. Equivalently, the holomorphic structure is given by duality as in Definition \ref{def:holrealdef}. 
Under the identification $(T^{(1,0)} \M)_m \simeq (T^{(0,1)} \M^{(0)})_m \otimes \mathcal{O}(1)$, and in the trivialization $e$ of $\mathcal{O}(1)|_{\CC}$, a holomorphic section of $(T^{(1,0)} \M)_m$ is described by a section $\zeta \to v(\zeta)$ for 
\be v(\zeta) = a +\zeta b \label{eq:holsec} \ee
for some $a, b \in (T^{(1,0)} \M^{(0)})_m$.
In order to relate $s$ to $v$, 
we use the pairing between $(T^*_{(1,0)} \M)_m$ and $(T^{(1,0)}\M)_m$. Since
$(T^*_{(1,0)}\M)_m$ is naturally an holomorphic $\mathcal{O}(-1)$ bundle via the identification  $T_{(1,0)}^* \M^{(\zeta)} = (1-\zeta J_1') T^*_{(1,0)}\M^{(0)}$ (Proposition \ref{prop:iso0tozeta}),  let $\sigma$ be a section of $(T^*_{(1,0)}\M)_m \subset (T^*_\CC \M)_m$ and let $\tau$ be a section of the trivial bundle  $(T^*_{(1,0)}\M^{(0)})_m \times \CC \to \CC$ related by $\sigma = (1-\zeta J_1')\tau$.
Consequently, $v$ is related to $s$ by
\[\sigma(s) = \tau(v).\]
\bigskip

We now seek to characterize real holomorphic sections in terms of the section $v$. We first relate $s, t$ and $v$. Plugging in (a) $\sigma=(1-\zeta J_1')\tau$ and (b) $s=(1+\overline{\zeta} J_1)t$ into $\tau(v)=\sigma(s)$, we find that 
\be \tau(v)=\sigma(s) \overset{(a)}= \tau((1 - \zeta J_1) s) \overset{(b)}= \tau(((1+|\zeta|^2) - 2i \Imag\zeta J_1)t) = \tau((1+|\zeta|^2)t) \ ; \label{eq:dual1} \ee
the last equality above used the fact that $J_1$ maps $T^{(1,0)}\M^{(0)}$ to $T^{(0,1)}\M^{(0)}$, but $\tau\in T^*_{(1,0)}\M^{(0)}$. Since 
this holds 
for all $\tau$, we have
\be v = (1+|\zeta|^2)t \ . \label{eq:ttov} \ee
Therefore, since $s=(1+\overline{\zeta} J_1)t$,
\be s(\zeta) = \frac{1}{1+|\zeta|^2} (1+\bar\zeta J_1) v(\zeta)  \label{eq:stov1}\ . \ee

Recall that, in terms of $s$, a real section is one that satisfies $s(\zeta) = \overline{s(-1/\overline{\zeta})}.$ We translate this to $v$ using \eqref{eq:stov1}. 
For $\zeta\in \CC^\times$, we now compute
\be \overline{s(-1/\bar\zeta)} = \frac{-\zeta}{1+|\zeta|^2} (1+\bar\zeta J_1) J_1 \overline{ v(-1/\bar\zeta)} . \ee
hence, in terms of $v$, a real section is one for which $v$ satisfies 
\be v(\zeta) = - \zeta J_1 \overline{v(-1/\bar\zeta)} \ . \label{eq:vReal} \ee
Using the holomorphic condition in \eqref{eq:holsec} and condition that $s$ be in $T^{(1,0)} \M$ in \eqref{eq:ttov}, real holomorphic sections of $T^{(1,0)} \M \to \PP^1$ in this language are those sections of the form
\be v(\zeta) = a - \zeta J_1 \bar a \ , \label{eq:realDual} \ee
where $a\in (T^{(1,0)}\M^{(0)})_m$.

\bigskip

In particular, real holomorphic sections are determined by their value at $\zeta=0$.
\footnote{Of course, there is nothing special about $\zeta=0$, and so one concludes that they are determined by their value at any single point. Indeed, for any $\zeta_0\in \CC$, by defining $a'=a - \zeta_0 J_1 \bar a$ (with inverse $a=\frac{1}{1+|\zeta_0|^2}(a'+\zeta_0 J_1 \overline{a'})$) we can rewrite \eqref{eq:realDual} as
\be v(\zeta) = a' - \frac{\zeta-\zeta_0}{1+|\zeta_0|^2} (-\bar\zeta_0 a' + J_1 \overline{a'}) \ . \ee

This proof that nonzero real holomorphic sections vanish nowhere, unlike that of \cite{hitchin:hkSUSY}, does not make use of a positive-definite metric, and so enables the reconstruction of hypercomplex manifolds from their twistor spaces.} 
This gives the inverse of the map in the statement of the lemma, showing that it is an isomorphism.
\end{proof}

\section{From a family of holomorphic symplectic structures to a pseudo-hyper-K\"ahler structure}

In order to specify a pseudo-hyper-K\"ahler structure on a real manifold $\M$, we reverse many of the above constructions. This culminates in Theorem \ref{thm:twist}, where we construct a pseudo-hyper-K\"ahler structure from a $\mathbb{P}^1_\zeta$-family of holomorphic symplectic forms $\varpi(\zeta)$ of a particular shape on the real manifold $\M$ (see Definition \ref{def:holoSympDef}).  In some of the preparatory material, we follow \cite{neitzke:higgsNotes}.

\begin{definition}[
] \label{def:holoSympDef}
A \emph{holomorphic symplectic form on a real manifold} $\M$ is a closed 2-form $\Omega$, which we regard as a map $\Omega: T_\CC\M\to T^*_\CC\M$, such that $T_\CC \M = \ker\Omega \oplus \ker \bar\Omega$.
\end{definition}
\noindent This definition is related to usual definition of holomorphic symplectic form on a complex manifold as follows:
\begin{proposition}[c.f. Proposition 2.56 of \protect\cite{neitzke:higgsNotes}]
Let $\Omega$ be a holomorphic symplectic form on a real manifold $\M$. Then, there is a unique complex structure $J$ on $\M$ such that $\Omega$ is a holomorphic symplectic form on the complex manifold $(\M,J)$.

Conversely, a holomorphic symplectic form on a complex manifold is also a holomorphic symplectic form on the underlying real manifold in the sense of Definition \ref{def:holoSympDef}.
\end{proposition}
\begin{proof} 
The unique almost complex structure with respect to which $\Omega$ is a $(2,0)$-form that defines an isomorphism $\Omega: T^{(1,0)}\to T^*_{(1,0)}$ is given by declaring $\ker\Omega = T^{(0,1)}\M$ and $\ker\bar\Omega = T^{(1,0)}\M$. To see that this is integrable, let $v$ and $w$ be $(1,0)$-vector fields. Then, for any vector field $x$ we have
\be 0 = d\bar\Omega(v,w,x) = - \bar\Omega([v,w],x) \ . \ee
This shows that $[v,w]$ is also of type $(1,0)$.

The converse is clear.
\end{proof}

\noindent A convenient means of proving that a 2-form $\Omega$ is holomorphic symplectic is the following:
\begin{proposition} \label{prop:holoVol}
A closed 2-form $\Omega$ on a real $4r$-manifold $\M$ is a holomorphic symplectic form if and only if $\Omega^r\wedge\bar\Omega^r$ vanishes nowhere and $\dim_\CC \ker\Omega \ge 2r$.
\end{proposition}
\begin{rem}
The condition $\dim_\CC \ker \Omega\ge 2r$ is particularly convenient because the nullity of $\Omega$ is upper semicontinuous.
\end{rem}
\begin{proof}
If $\Omega$ is a holomorphic symplectic form then $\Omega^r\wedge \bar\Omega^r$ is non-degenerate, as is clear from the holomorphic Darboux theorem, and $\dim_\CC \ker\Omega = 2r$. Conversely, if $\Omega^r\wedge \bar\Omega^r$ vanishes nowhere then $\ker\Omega\cap\ker\bar\Omega=\{0\}$, so that there is an injective map $\ker\Omega\oplus\ker\bar\Omega\hookrightarrow T_\CC\M$, and $\dim\ker\Omega\ge 2r$ implies that this is also surjective.
\end{proof}

\begin{proposition} \label{prop:realSymp}
If $\Omega=\omega_1+i\omega_2$ is a holomorphic symplectic form on a complex manifold $(\M,J)$, where $\omega_1,\omega_2$ are real 2-forms, then $\omega_1,\omega_2$ are symplectic forms and $\omega_2(v,w)=-\omega_1(Jv,w)=-\omega_1(v,Jw)$.
\end{proposition}
\begin{proof}The symplectic forms $\omega_1$ and $\omega_2$ are clearly closed.

Since $\Omega$ is a $(2,0)$-form, we have
\be \omega_1(Jv,w)+i\omega_2(Jv,w)=\Omega(Jv,w)=i \Omega(v,w)=-\omega_2(v,w)+i\omega_1(v,w) \ee
for all vectors $v,w$, and so if $v,w$ are real then taking real parts gives $\omega_2(v,w)=-\omega_1(Jv,w)$. An identical argument gives $\omega_2(v,w)=-\omega_1(v,Jw)$. These results then extend by complex linearity to all vectors $v,w$.

To see that $\omega_1$ is non-degenerate, we observe that
\be \omega_1(v,\cdot)=0\Rightarrow \Omega(v,\cdot)=\omega_1(v,(1-iJ)\cdot)=0 \Rightarrow v\in T^{(0,1)}\M \ . \ee
An identical argument using $\bar\Omega$ shows that $v\in T^{(1,0)}\M$, and hence $v=0$. Similar arguments work for $\omega_2$, since $\Omega(v,\cdot)=\omega_2(v,(J+i)\cdot)$. (Alternatively, we can use the fact that $e^{i\theta}\Omega$ is a holomorphic symplectic form for any $\theta\in\RR$, and so in particular for $\theta=-\pi/2$.) 
\end{proof}

As promised in the introduction, we now demonstrate that $\kappa$ automatically exists once you have a $\mathbb{P}^1$-family of holomorphic symplectic forms $\varpi(\zeta)$.
\begin{proposition}\label{prop:varpiRealJ}
Let $\omega_+$ be a holomorphic symplectic form on a manifold $\M$ with associated complex structure $J^{(0)}$.
Let
\be \varpi(\zeta) = - \frac{i}{2\zeta} \omega_+ + \omega_3 - \frac{i}{2} \zeta \omega_- \label{eq:varpiExp} \ee
be a family of holomorphic symplectic forms on $\M$ labeled by $\zeta\in \CC^\times$, with $\omega_- = \bar \omega_+$ and $\bar\omega_3 = \omega_3$.
Then,
\begin{enumerate}[(a)]
\item the complex structures defined by these holomorphic symplectic forms satisfy $J^{(-1/\bar\zeta)} = - J^{(\zeta)}$ for all $\zeta\in \PP^1$. 
\item $\omega_3 \in \Omega^{1,1}(\M^{(0)})$ and, moreover, is non-degenerate;
\item there is an isomorphism $\kappa: T^{(0,1)} \M^{(0)} \to T^{(1,0)} \M^{(0)}$ such that 
$T^{(0,1)} \M^{(\zeta)} = (1+\zeta \kappa) T^{(0,1)} \M^{(0)}$ for all $\zeta\in\CC$;
explicitly, using the non-degeneracy of $\omega_+$, $\kappa$ is defined by 
\begin{equation} \iota_{\kappa v} \omega_+ = -2 i  \iota_v \omega_3. \end{equation}
Moreover, $\kappa$ satisfies 
\begin{equation} \iota_{\kappa v} \omega_3 = \frac{i}{2}  \iota_v \omega_-. \end{equation}
\end{enumerate}
\end{proposition}

\begin{proof}
\begin{enumerate}[(a)]
\item For all $\zeta\in \CC^\times$ we have $\overline{\varpi(-1/\bar\zeta)} = \varpi(\zeta)$, so that $T^{(0,1)}\M^{(\zeta)} = T^{(1,0)}\M^{(-1/\bar\zeta)}$. Since $\omega_-=\bar\omega_+$, we also have $T^{(0,1)}\M^{(0)}=T^{(1,0)}\M^{(\infty)}$.
\item We claim that $\omega_3 \in \Omega^{1,1}(\M^{(0)})$. Since $\dim_\CC \M=2r$,  consider the wedge product of $(r+1)$ copies of $\omega(\zeta)$, here denoted $\omega(\zeta)^{(r+1)} \in \Omega^{2(r+1),0}(\M^{(0)})=\{0\}$.
Expanding this in powers of $\zeta$, 
\[ \omega(\zeta)^{(r+1)} = \left(- \frac{i}{2 \zeta} \omega_+\right)^{(r+1)} + (r+1) 
 \left(- \frac{i}{2 \zeta} \omega_+\right)^{r} \wedge \omega_3 + \cdots  + (r+1) 
 \left(- \frac{i}{2}\zeta \omega_-\right)^{r} \wedge \omega_3  +  \left(- \frac{i}{2}\zeta \omega_-\right)^{(r+1)},\]
 each term individually vanishes; thus,  $\omega_+^r \wedge \omega_3=0$ and $\omega^r_- \wedge \omega_3=0$. Since $0=\omega_+^r \wedge \omega_3^{(0,2)} \in \Omega^{(2r, 2)}(\M^{(0)})$, it follows from the non-degeneracy of $\omega_+$ and hence the non-vanishing of $\omega_+^r$ that $\omega_3^{(0,2)}=0$; a similar computation yields $\omega_3^{(2,0)}=0$. Thus, $\omega_3  \in \Omega^{1,1}(\M^{(0)})$, as claimed.
 
We note that if $\zeta = \exp(i \theta)$, $\varpi(\zeta)=\omega_3 - i (\cos \theta \omega_1 + \sin \theta \omega_2)$. From Proposition \ref{prop:realSymp}, the non-degeneracy of $\varpi(\zeta)$ implies its real part $\omega_3$ is also non-degenerate. 
 
\item First, we observe that $T^{(0,1)} \M^{(\zeta)} \cap T^{(1, 0)} \M^{(0)}=\{0\}$ for $\zeta \in \CC$ : if $v \in T^{(0,1)} \M^{(\zeta)} \cap T^{(1, 0)} \M^{(0)}$, then 
\be 0 = \iota_v \varpi(\zeta) =  - \frac{i}{2\zeta} \iota_v \omega_+ + \iota_v\omega_3.\ee
The $(1,0)$ and $(0,1)$-forms with respect to $\M^{(0)}$ separately vanish, i.e. $- \frac{i}{2\zeta} \iota_v \omega_+=0$ and $\iota_v\omega_3=0$. Since $\omega_+: T^{(1, 0)} \M^{(0)} \to T^*_{(1,0)} \M^{(0)}$ is non-degenerate, $v=0$.
It follows that $T^{(0,1)} \M^{(\zeta)} \subset T_\CC \M =T^{(0,1)} \M^{(0)}\oplus T^{(1,0)} \M^{(0)}$ maps surjectively onto $T^{(0,1)} \M^{(0)}$. Reversing this, $T^{(0,1)} \M^{(\zeta)}$ is a graph over $T^{(0,1)} \M^{(0)}$ (see Figure \ref{fig:graph}).
\begin{figure}[h!]
    \begin{centering} 
    \includegraphics[height=1.2in]{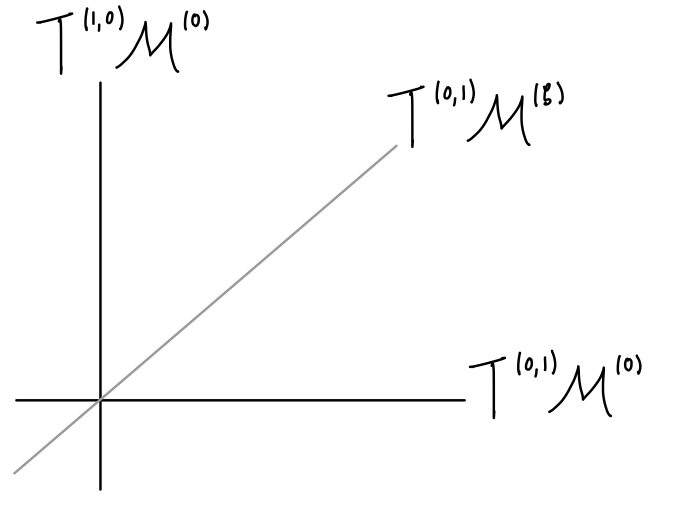} 
    \caption{\label{fig:graph} 
     $T^{(0,1)} \M^{(\zeta)}$ is the graph of $\kappa(\zeta)$ over $T^{(0,1)} \M^{(0)}$.
    }
    \end{centering}
    \end{figure}
Hence, 
for any fixed parameter $\zeta \in \CC$ there is a map $\kappa(\zeta): T^{(0,1)} \M^{(0)} \to T^{(1,0)} \M^{(0)}$ such that $(1 + \kappa(\zeta)): T^{(0,1)} \M^{(0)} \to T^{(0,1)} \M^{(\zeta)}$.  A priori, this is not necessarily an isomorphism.  
Let $v \in T^{(0,1)} \M^{(0)}$.
Then the interior product of $\varpi(\zeta) \in \Omega^{2,0}\M^{(\zeta)}$ with a vector $v + \kappa(\zeta) v \in T^{(0,1)} \M^{(\zeta)}$ vanishes, i.e.
\begin{align}
0&=\iota_{v + \kappa(\zeta) v}\varpi(\zeta) \\ \nonumber
&=  \left(- \frac{i}{2\zeta} \iota_{\kappa(\zeta) v}\omega_+ + \iota_v\omega_3\right) + \left(\iota_{\kappa(\zeta) v} \omega_3- \frac{i}{2} \zeta \iota_v\omega_- \right); 
\end{align}
the $(1,0)$ and $(0,1)$-forms components on $\M^{(0)}$ separately vanish, i.e. 
\begin{align}\label{eq:kappa}
\iota_{\kappa(\zeta) v} \omega_+ &= -2 i \zeta \iota_v \omega_3\\ \nonumber 
\iota_{\kappa(\zeta) v} \omega_3 &= \frac{i}{2} \zeta \iota_v \omega_-.
\end{align}
Now, we allow $\zeta$ to vary. Since the $\CC^\times$-family $\kappa(\zeta)$ satisfy \eqref{eq:kappa}, it follows that $\kappa(\zeta) = \zeta \kappa$  for $\kappa= \kappa(1)$. Thus, we have a map  $\kappa: T^{(0,1)} \M^{(0)} \to T^{(1,0)} \M^{(0)}$ such that 
$(1+\zeta \kappa) : T^{(0,1)} \M^{(0)} \to T^{(0,1)} \M^{(\zeta)} $ for all $\zeta\in\CC$. Lastly, since $\omega_3: T^{(0,1)} \M^{(0)} \to T^*_{(1,0)} \M^{(0)}$ is non-degenerate and $\omega_+$ is non-degenerate, the  \eqref{eq:kappa} implies $\kappa$ is an isomorphism.
\end{enumerate}
\end{proof}

Now, we state our main theorem, making use of the above isomorphism $\kappa$.

\begin{theorem} \label{thm:twist}\hfill
\begin{enumerate}[(a)]
\item Let $\M$ be a manifold with a $\PP^1$-worth of complex structures $J^{(\zeta)}$ such that $J^{(-1/\bar\zeta)}=-J^{(\zeta)}$. Suppose also that there is an isomorphism $\kappa: T^{(0,1)}\M^{(0)}\to T^{(1,0)}\M^{(0)}$ such that $T^{(0,1)} \M^{(\zeta)} = (1+\zeta \kappa) T^{(0,1)} \M^{(0)}$ for all $\zeta\in\CC$.

Then, defining $J_1=J^{(i)}$, $J_2=J^{(-1)}$, and $J_3=J^{(0)}$, it follows that (i) $\kappa=J_1|_{T^{(0,1)}\M^{(0)}}$ and  (ii) the triplet of complex structures satisfies the unit quaternion algebra in 
\eqref{eq:quat}.

\item Let \be \varpi(\zeta) = - \frac{i}{2\zeta} \omega_+ + \omega_3 - \frac{i}{2} \zeta \omega_- \ee be a family of holomorphic symplectic forms such that $\omega_+$ is holomorphic symplectic and $\omega_- = \bar \omega_+$ and $\omega_3=\bar \omega_3$. 
Then,
$\M$ is a pseudo-hyper-K\"ahler manifold with pseudo-K\"ahler forms $(\omega_1,\omega_2,\omega_3)$ corresponding to the complex structures $(J_1,J_2,J_3)$, where $\omega_\pm = \omega_1 \pm i \omega_2$.

\item If the family of holomorphic symplectic forms $\varpi(\zeta)$ in (2) arises from a pseudo-hyper-K\"ahler structure on $\M$, then the pseudo-hyper-K\"ahler structure produced in (2) agrees with the original pseudo-hyper-K\"ahler structure on $\M$.
\end{enumerate}
\end{theorem}
\begin{proof}\hfill
\begin{enumerate}[(1)]\item We begin by extending $\kappa: T^{(0,1)}\M^{(0)}\to T^{(1,0)}\M^{(0)}$ to a real automorphism $\kappa: T_\CC \M\to T_\CC \M$ which exchanges $T^{(1,0)}\M^{(0)}$ and $T^{(0,1)}\M^{(0)}$. (Explicitly, if $v\in T^{(1,0)}\M^{(0)}$ then $\kappa v := \overline{\kappa \bar v}$.) We then observe that $1+\zeta\kappa'$ maps $T^*_{(1,0)}\M^{(\zeta)}$ to $T^*_{(1,0)}\M^{(0)}$; the identification $T^{(0,1)}\M^{(\zeta)} = T^{(1,0)}\M^{(-1/\bar\zeta)} = \overline{T^{(0,1)}\M^{(-1/\bar\zeta)}} = (1-\zeta^{-1}\kappa)T^{(1,0)}\M^{(0)}$ shows that $1-\zeta^{-1}\kappa'$ maps $T^*_{(1,0)}\M^{(\zeta)}$ to $T^*_{(0,1)}\M^{(0)}$. In particular, $1+i\kappa'$ maps $T^*_{(1,0)}\M^{(i)}$ to both $T^*_{(1,0)}\M^{(0)}$ and $T^*_{(0,1)}\M^{(0)}$, and hence vanishes on $T^*_{(1,0)}\M^{(i)}$. So, $\kappa'=J_1'$ on $T^*_{(1,0)}\M^{(i)}$, and hence by reality of both $\kappa'$ and $J_1'$ we have $\kappa=J_1$ on all of $T_\CC\M$.

Similarly, $J_1J_3|_{T^{(0,1)}\M^{(0)}} = -i J_1|_{T^{(0,1)}\M^{(0)}} = -J_3 J_1|_{T^{(0,1)}\M^{(0)}}$, since $J_1$ maps $T^{(0,1)}\M^{(0)}$ to $T^{(1,0)}\M^{(0)}$, and because $J_1$ and $J_3$ are real we in fact have $J_1 J_3=-J_3 J_1$ on all of $T_\CC\M$. Then, $J_3 J_1|_{T^{(0,1)}\M^{(-1)}} = J_3 J_1(1-J_1)|_{T^{(0,1)}\M^{(0)}} = -J_1(1+J_1)J_3 |_{T^{(0,1)}\M^{(0)}} = -i (1-J_1)|_{T^{(0,1)}\M^{(0)}} = J_2|_{T^{(0,1)}\M^{(-1)}}$ (plus reality) implies that $J_3 J_1 = J_2$. So, we have \eqref{eq:quat}. Finally, we observe that $J^{(\zeta)}|_{T^{(0,1)}\M^{(\zeta)}} = \sum_\alpha c_\alpha^{(\zeta)} J_\alpha |_{T^{(0,1)}\M^{(\zeta)}}$, and hence by reality we have $J^{(\zeta)} = \sum_\alpha c_\alpha^{(\zeta)} J_\alpha$. From this, we obtain \eqref{eq:holo2}, so that $(T^*_{(1,0)}\M^{(0)})_m\otimes\Oo(-1)\cong (T^*_{(1,0)}\M)_m$.

\item 
For each $m\in \M$, we have an isomorphism $(\wedge^2 T^*_{(1,0)}\M^{(0)})_m \otimes \Oo(-2)\cong (\wedge^2 T^*_{(1,0)}\M)_m$ which implies that 
\be \varpi(\zeta)(v, w)=-\frac{i}{2 \zeta} \Omega( (1-\zeta J_1) v, (1-\zeta J_1)w) \qquad  \forall v, w \ee
for some $\Omega\in (\wedge^2 T^*_{(1,0)}\M^{(0)})_m$ (since $\varpi(\zeta)|_m$ is a holomorphic section of $(\wedge^2 T^*_{(1,0)}\M)_m \otimes \Oo(2)$ and the only holomorphic sections of the trivial bundle $(\wedge^2 T^*_{(1,0)}\M^{(0)})_m$ are constants). 
For convenience, we will instead use the compact notation
\be \varpi(\zeta) = -\frac{i}{2\zeta}(1-\zeta J_1')\otimes (1-\zeta J_1') \Omega\ee 
that omits the vectors $v, w$.
Comparing the $\zeta^{-1}$ terms shows that $\Omega=\omega_+|_m$, so that
\be \varpi(\zeta) = - \frac{i}{2\zeta}(1-\zeta J_1')\otimes (1-\zeta J_1') \omega_+ \ . \label{eq:varpiBundle} \ee
We also have
\be \varpi(\zeta)=-\frac{i}{2\zeta}(J_1'+\zeta)\otimes(J_1'+\zeta)\omega_- \ , \label{eq:varpiRealBundle} \ee
thanks to the reality condition $\overline{\varpi(-1/\bar\zeta)} = \varpi(\zeta)$. In particular, the leading behavior near $\zeta=0$ shows that $\omega_+ = J_1'\otimes J_1' \omega_-$.

With this, we are ready to define a metric on $\M$:\footnote{Note that this agrees with the special case of \eqref{eq:holoG} with $\zeta=0 e^{-i\pi/2}\in \widetilde{\PP^1}$, the real blowup of $\mathbb{P}^1$ at $0, \infty$.}
\be g := \frac{1\otimes J_1'-J_1'\otimes 1}{2} \omega_+ \ . \label{eq:twistMet} \ee
This is a symmetric tensor. A short computation using \eqref{eq:varpiBundle} yields that\footnote{As an alternative to some of these computations, one may use $\omega_- = J_1'\otimes J_1' \omega_+$ to show that $g=(1\otimes J_1')\Real \omega_+ = (1\otimes J_2') \Imag \omega_+$.}
\be g = (1\otimes J_1')\Imag \varpi(-1) = (1\otimes J_2')\Imag \varpi(-i) = (1\otimes J_3') \Real \varpi(-1) \ , \ee
so that Proposition \ref{prop:realSymp} implies that $g$ is a pseudo-K\"ahler metric in complex structures $J_1,J_2,J_3$ with associated pseudo-K\"ahler forms $\omega_1,\omega_2,\omega_3$. For example, we have
\begin{align}
(1\otimes J_2')\Imag\varpi(-i) &= \frac{1}{2i}(1\otimes J_2')(\varpi(-i)-\varpi(i)) \nonumber \\
&= \frac{1}{4i}(1\otimes J_1')(1\otimes J_3')((1+iJ_1')\otimes (1+iJ_1')+(1-iJ_1')\otimes (1-iJ_1'))\omega_+ \nonumber \\
&= \frac{1}{4i}(1\otimes J_1')((1+iJ_1')\otimes (1-iJ_1')+(1-iJ_1')\otimes (1+iJ_1'))(1\otimes J_3')\omega_+ \nonumber \\
&= \frac{1}{4} (1\otimes J_1')((1+iJ_1')\otimes (1-iJ_1')+(1-iJ_1')\otimes (1+iJ_1')) \omega_+ \nonumber \\
&= \frac{1}{4}((1+iJ_1')\otimes (i+J_1') + (1-iJ_1')\otimes (-i+J_1'))\omega_+ \nonumber \\
&= \half (1\otimes J_1' - J_1' \otimes 1)\omega_+ \nonumber \\
&= g \ ;
\end{align}
the other computations are similar.
\item It is clear that the metric $g$ in \eqref{eq:twistMet} is the original pseudo-hyper-K\"ahler metric.
\end{enumerate}
\end{proof}

We briefly compare this with the usual statement of the twistor theorem:
\begin{theorem}[Construction of hyper-K\"ahler manifolds from twistor spaces, {\cite[Theorem 3.3]{hitchin:hkSUSY}}] \label{thm:HKLR}Let $\Z^{4r+2}$ by a real manifold 
carrying all the structures of Proposition \ref{prop:twistorprop}. Let $\widetilde{\M}$ be the space of all real holomorphic sections of $\Z$ having normal bundle isomorphic to $\mathcal{O}(1)^{\oplus 2 r}$. Then, $\widetilde{\M}$ admits the structure of a smooth manifold and a
pseudo-hyper-K\"ahler structure. 
\end{theorem}
One of the key steps in Hitchin--Karlhede--Lindstr\"om--Ro\v{c}ek is the identification of $\mathcal{Z}\simeq\mathcal{M} \times \mathbb{P}^1$ as real manifolds via the space of real holomorphic sections of $\Z$ having normal bundle isomorphic to $\mathcal{O}(1)^{\oplus 2r}$. 
However, 
in our concrete situation, we begin with an identification  $\mathcal{Z}\simeq\mathcal{M} \times \mathbb{P}^1$ as real manifolds and a canonical real structure. The isomorphism in Lemma \ref{lem:realHol}
is the key to understanding why the metric \eqref{eq:twistMet} on $\M$ is the same as that defined in \cite{hitchin:hkSUSY}:
the metric of \cite[(3.103)]{hitchin:hkSUSY}, expressed as a bilinear form on the space of real holomorphic sections of $(T^{(1,0)}\M)_m$, denoted $H^0_\RR((T^{(1,0)}\M)_m)$, is
\be g(a-\zeta j a, b-\zeta j b) = \half(\omega_+(a, j b)-\omega_+(ja, b)) \ , \label{eq:twistMetHKLR} \ee
where $j$ is an antilinear operator on $(T^{(1,0)}\M^{(0)})_m$ which squares to $-1$ such that real holomorphic sections are those of the form $a-\zeta j a$.
As we showed in Lemma \ref{lem:realHol}, if $\Z$ and its family of holomorphic symplectic forms $\varpi$ comes from a pseudo-hyper-K\"ahler metric, then  \eqref{eq:realDual} shows that $j=J_1 \overline{\cdot}$. With this identification, the bilinear form in \eqref{eq:twistMetHKLR} agrees with the bilinear form we defined in \eqref{eq:twistMet}, which we state as a corollary for emphasis.
\begin{corollary}[Corollary to Lemma \ref{lem:realHol} \& Theorem \ref{thm:twist}c]\label{cor:match}
The pseudo-hyper-K\"ahler structure produced in Theorem \ref{thm:twist}b coincides with the pseudo-hyper-K\"ahler structure defined by Hitchin--Karlhede--Lindstr\"om--Ro\v{c}ek in \cite{hitchin:hkSUSY}. 
\end{corollary}
\bibliography{refs}

\end{document}